\documentclass[aps, twocolumn]{revtex4}

\usepackage{bm, amsfonts, mathtools}
\usepackage{graphicx, color, wasysym}
\usepackage{textcomp}
\usepackage{amsmath, amssymb, amsthm, latexsym, epsfig}
\usepackage{appendix}

\def\ulamek#1#2{\mbox{\normalfont$\frac{#1}{#2}$}}

\DeclareMathOperator{\okr}{{\stackrel{{\scriptscriptstyle{\mathsf{def}}}}{=}}}
\DeclareMathOperator{\D}{d\!}
\DeclareMathOperator{\E}{e} 
\DeclareMathOperator{\I}{i}
   \DeclareMathOperator{\RE}{\mathfrak{Re}}

\theoremstyle{plain}
\newtheorem{thm}{Theorem}

\theoremstyle{definition}

\newtheorem{exa}[thm]{{\it Example}}

\theoremstyle{remark}
\newtheorem{rem}[thm]{{\it Remark}}

\begin{document} 

\begin{center}
\large{This paper is now published (in revised form) in Fract. Calc. Appl. Anal. {\bf 22}(5) (2019) 1284--1306, \\
\medskip
DOI: 10.1515/fca-2019-0068; \\
\ \\ \ \\ \ \\
\medskip
and is available online at http://www.degruyter.com/view/j/fca .}
\end{center}

\ \\ \ \\

\title[Some results on the complete monotonicity of the Mittag-Leffler functions of Le Roy type.]{\bigskip\bigskip\bigskip Some results on the complete monotonicity of the Mittag-Leffler functions of \\ Le Roy type.}

\author{K. G\'{o}rska}
\email{katarzyna.gorska@ifj.edu.pl}
\affiliation{H. Niewodnicza\'{n}ski Institute of Nuclear Physics, Polish Academy of Sciences, 
ul.Eljasza-Radzikowskiego 152, PL 31342 Krak\'{o}w, Poland}

\author{A. Horzela}
\email{andrzej.horzela@ifj.edu.pl}
\affiliation{H. Niewodnicza\'{n}ski Institute of Nuclear Physics, Polish Academy of Sciences, 
ul.Eljasza-Radzikowskiego 152, PL 31342 Krak\'{o}w, Poland}

\author{R. Garrappa}
\email{roberto.garrappa@uniba.it}
\affiliation{Department of Mathematics, University of Bari "Aldo Moro", \\ Via Orabona n. 4 - 70125 Bari, Italy - Member of the INdAM Research group GNCS}

\begin{abstract}

The paper by R. Garrappa, S. Rogosin, and F. Mainardi, entitled {\em On a generalized three-parameter Wright function of the Le Roy type} and published in [Fract. Calc. Appl. Anal. {\bf 20} (2017) 1196-1215], ends up leaving the open question concerning the range of the parameters $\alpha, \beta$ and $\gamma$ for which Mittag-Leffler functions of Le Roy type $F_{\alpha, \beta}^{(\gamma)}$ are completely monotonic. Inspired by the 1948 seminal H. Pollard's paper  which provides the proof of the complete monotonicity of the one parameter Mittag-Leffler function, the Pollard approach is used to find the Laplace transform representation of $F_{\alpha, \beta}^{(\gamma)}$ for integer $\gamma = n$ and rational $0 < \alpha \leq 1/n$. In this way it is possible to show that  Mittag-Leffler functions of Le Roy type are completely monotone for $\alpha = 1/n$ and $\beta \geq (n+1)/(2n)$ as well as for rational $0 < \alpha \leq 1/2$, $\beta = 1$ and $n=2$.  For further integer values of $n$ the complete monotonicity is tested numerically for rational $0< \alpha < 1/n$ and various choices of $\beta$. The obtained results suggest that for the complete monotonicity the condition $\beta \geq (n+1)/(2n)$ holds for any value of $n$. 
\end{abstract}

\maketitle
 
\section{Introduction}\label{sec1}
\vspace*{-16pt}
 
 \definecolor{red}{rgb}{1, 0, 0}

The Mittag-Leffler (ML) functions of the Le Roy type (denoted here, for shortness, as {\em MLR functions}) are defined in \cite{RGarrappa17} as
\begin{multline*} 
F_{\alpha, \beta}^{(\gamma)}(z) = \sum_{r=0}^{\infty} \frac{z^{r}}{[\Gamma(\beta + \alpha r)]^{\gamma}}, \\ z\in\mathbb{C}, \quad \alpha, \beta, \gamma\in\mathbb{C}, \,\, \Re(\alpha) > 0,
\end{multline*}
and treated as generalizations of the Mittag-Leffler (ML) function $E_{\alpha, \beta}(z)=\sum_{r=0}^{\infty} z^{r/\Gamma(\beta + \alpha r)}$. These are entire functions of the complex variable $z$ for all values of the parameters such that $\Re(\alpha)>0$, $\beta \in \mathbb{R}$, and $\gamma >0$.

In this paper we consider the MLR functions {\em in the case of integer positive $\gamma=n$},
\begin{multline}\label{27/02-1}
F_{\alpha, \beta}^{(n)}(-x) = \sum_{r=0}^{\infty} \frac{(-x)^{r}}{[\Gamma(\beta + \alpha r)]^{n}}, \quad x\in\mathbb{R}, \\ \alpha >0, \,\, \beta\in\mathbb{R}, \,\, n = 1, 2, \ldots,
\end{multline}
For particular choices of the parameters they return more widely-known special functions: the one parameter ML function $E_{\alpha}(-x)$ and the two parameter ML function $E_{\alpha, \beta}(-x)$ for $n = 1$ and $n = \beta = 1$ respectively; the Bessel function of the first kind $J_{0}(2\sqrt{x})$ when $\alpha = \beta = 1$ and $n=2$; the standard Le Roy function when $\alpha = 1$, $\beta = 2$ and arbitrary integer $n$. 

For many years several efforts have been payed to study the complete monotonicity (CM) of the Mittag-Leffler functions which has been, however, proved just for some of the mentioned special cases but not for the more general case for which CM has been only conjectured \cite{RGarrappa17}. 

We recall that a function $f(x)$, $x\in D$, is CM if $f \in {\mathcal C}^{\infty}(D)$ and all its $n$-th derivatives satisfy the inequalities \cite{RLSchilling12,DVWidder46}
\begin{equation}\label{5/03-1}
(-1)^{n} f^{(n)}(x) \geq 0, \qquad n = 0, 1, 2, \ldots.
\end{equation}


According to the Bernstein theorem \cite[Theorem 12a]{DVWidder46}, the key property of CM functions is that the are uniquely representable as the Laplace transforms of non-negative weight functions $F(p)$ supported on $p\in[0, \infty)$, i.e.
\begin{equation}\label{5/03-2}
f(x) = \int_{0}^{\infty} \E^{- x p} F(p) \D p.
\end{equation}

The CM character of the one-parameter ML function $E_{\alpha}(-x)$, for $0 < \alpha < 1$, was first shown in \cite{HPollard48} and results for the two-parameter ML function were instead presented in \cite{VVAnh03, KSMiller97, WRSchneider96} for $0 < \alpha < 1$ and $\beta \geq \alpha$. In the recent paper \cite{KGorska18a} it is shown that also the three parameter ML function $E_{\alpha, \beta}^{\gamma}(-x)$ (often referred to as the Prabhakar function) is CM when $0 < \alpha < 1$, $\gamma > 0$, and $\beta \geq \alpha\gamma$, thus extending some results previously investigated in  \cite{FMainardi15, Tomovski} for the special case of the compound function $t^{\beta-1}E_{\alpha, \beta}^{\gamma}(-t^{\alpha})$ under the conditions $0 < \alpha < 1$, $\gamma > 0$, and $0 < \alpha\gamma \leq \beta \leq 1$. Additionally, in \cite{KSMiller01} the CM character was studied for more familiar functions, namely the negative power, the exponential, and the modified Bessel functions of the first and second kind.

Inspired by the fact that the special cases of the MLR functions for $n = 1$ are CM we expect that under some conditions the MLR functions should be CM for $n\neq 1$ as well. Indeed, this is the essence of the unanswered question which closes the paper \cite{RGarrappa17}: 

{\em To find conditions on the parameters $\alpha$, $\beta$, $\gamma$ for which the function $F_{\alpha, \beta}^{(\gamma)}(-x)$, $0 < x < -\infty$, is completely monotone}. 

In this paper, by exploiting the technique proposed 70 years ago in \cite{HPollard48}, we  provide a first answer to this question limited to the case of integer $\gamma = n$ and rational $0<\alpha < 1$ .  

The paper is organized as follows. In Sec. \ref{sec2} we recall various forms of the MLR function often refereed to, in the physical and mathematical literature, as $\alpha$-Mittag-Leffler function, multi-index Mittag-Leffler function, and the Mittag-Leffler function of vector index. A novelty is that we introduce {{a representation}} of the MLR function {{in terms of a}} finite sum of generalized hypergeometric functions. The Laplace transform evaluated for this type of MLR function yields s known formula \cite{RGarra13, RGarrappa17} which allows to reduce the value of the parameter $n$ labeling the MLR function to $n-1$ (see Appendix \ref{appC}). In Sec. \ref{sec3} we invert this Laplace transform and next use it $n$ times. That enables us to represent the MLR function with the rational parameter $\alpha=l/k$ ($0 < l/k < 1$) as the $n$ times nested integral of the two parameters ML function $E_{l/k, \beta}(-x)$. Using the contour integral representation of $E_{l/k, \beta}(-x)$ we arrive at the Laplace transform representation of the MLR function with the weight function $m_{l/k, \beta}(n; y)$. The non-negativity of $m_{l/k, \beta}(n; y)$ for $\beta \geq (n+1)/(2n)$ and $l n = k$ as well as for $n=2$, $2l < k$ and $\beta = 1$ is shown analytically in Sec. \ref{sec4} and illustrated for special values of parameters for which $m_{l/k, \beta}(n; y)$ can be transformed into known special functions. We also check,  using numerical evaluations, the nonnegativity of $m_{l/k, \beta}(n; y)$ for various values of $\beta$, $l/k = 3/7, 3/10$, and $n = 2, 3$. It turns out that the condition $\beta \geq (n+1)/(2n)$ is still in the game and if it is kept the function $m_{l/k, \beta}(n; y)$ is non-negative. The paper is concluded in Sec. \ref{sec5}. The appendixes \ref{appA} and \ref{appB} contain the list of the used formulas related to the high transcendental functions, especially the Meijer $G$ and the generalized hypergeometric functions. 

\section{The variety of the MLR functions}\label{sec2}

The MLR function for real $x$, $\alpha = \nu_{\rm G}$, $\beta = \gamma_{G}$, and $n=\alpha_{G} + 1$ (the subscript $G$ is used to emphasize the reference \cite[Def. 3.1]{RGarra13} from which its definition and formula are taken) is called the $\alpha$-Mittag-Leffler function $E_{\alpha_{\rm G}; \nu_{\rm G}, \gamma_{\rm G}}(x)$. For integer $\alpha_{\rm G}$ the $\alpha$-ML function is expressed as the generalized Wright function according to \cite[Eq. (3.7)]{RGarra13}
\begin{equation}\label{12/06-3}
F_{\alpha, \beta}^{(n+1)}(x) = E_{n; \alpha, \beta}(x) =  {_{1}\Psi_{n+1}}\bigg[x\Big\vert {(1,1) \atop \underbrace{(\beta, \alpha), \ldots, (\beta, \alpha)}_{n+1 \,\,\, \text{times}}}\bigg],  
\end{equation}
where
\begin{equation}\label{12/06-4}
{_{p}\Psi_{q}}\bigg[z\Big\vert {(a_{1}, A_{1}), \ldots, (a_{p}, A_{p}) \atop (b_{1}, B_{1}), \ldots, (b_{q}, B_{q})}\bigg] = \sum_{r=0}^{\infty} \frac{z^{r}}{r!} \frac{\prod_{i=1}^{p} \Gamma(a_{i} + A_{i} r)}{\prod_{i=1}^{q} \Gamma(b_{i} + B_{i} r)}
\end{equation}
with $a_{i}, b_{i}\in\mathbb{C}$, $A_{i}, B_{i} > 0$, and complex $z$. The MLR function can be also represented by the multi-index Mittag-Leffler function \cite{VKiryakova10} known also as the ML function of vector index \cite{YuLuchko99}. In \cite{VKiryakova10, YuLuchko99} this function is expressed either as the series or the Fox $H$ function 
\begin{align}\label{12/06-2}
\begin{split}
E^{(m)}_{(1/\rho_{i}), (\mu_{i})} & = \sum_{r=0}^{\infty} \frac{z^{r}}{\Gamma(\mu_{1} + r/\rho_{1}) \cdots \Gamma(\mu_{m} + r/\rho_{m})} \\
& = H^{1, 1}_{1, m+1}\bigg[-z\Big\vert {(0, 1) \atop (0, 1), (1- \mu_{i}, 1/\rho_{i})^{m}_{1}}\bigg].
\end{split}
\end{align} 
{{whose main properties are listed in the Appendix \ref{appA}}}.

The asymptotic behaviour of the MLR function for {{complex arguments}} as well as its integral representations in the complex domain {{are discussed}} in \cite{RGarrappa17}. Moreover, in \cite{TPogany18}, it has been recently {{given}} the integral representation of the MLR function for $\alpha$, $\beta$, $\gamma, x > 0$, and $\alpha + \beta \geq x_{0}$ where $x_{0}$ is the abscissa of the minimum of the Gamma function. This representation reads
\begin{multline*}
F_{\alpha, \beta}^{(\gamma)}(x) = \frac{1}{[\Gamma(\beta)]^{\gamma}} + \frac{x}{[\Gamma(\alpha+\beta)]^{\gamma}} + \frac{x^{2}}{1-x} + \\ \frac{\gamma x}{x-1} \int_{1}^{\infty} x^{\lfloor(\Gamma^{-1}(t) - \beta)/\alpha\rfloor} \frac{\D t}{t^{\gamma+1}},
\end{multline*}
where $\Gamma^{-1}$ is defined on the increasing branch of the inverse Gamma function in the right half-plane and $[x]$ denotes the integer part of a real $x$.

Another possibility to represent the MLR function is to express it as a finite sum of generalized hypergeometric functions. This may be obtained  if one assumes the parameter $\alpha$ to be rational, $\alpha = l/k$, $0 < l/k < 1$,  and uses in the Eq. \eqref{27/02-1} the so-called splitting formula
\begin{equation}\label{6/06-1}
\sum_{r=0}^{\infty} a_{r} = \sum_{j=0}^{m-1} \sum_{r=0}^{\infty} a_{m r + j},
\end{equation} 
according to which the series of $a_{r}$ splits into $m$ sums with the terms $a_{mr}$, $a_{mr + 1}$, $\ldots$, $a_{mr + j}$. Thus, we can rewrite $F_{l/k, \beta}^{(n)}(z)$ as
\begin{equation}\label{6/06-2}
F_{\frac{l}{k}, \beta}^{(n)}(z) = \sum_{j=0}^{k-1} \frac{z^{j}}{[\Gamma(\beta + \frac{l}{k}j)]^{n}} \sum_{r=0}^{\infty} \frac{z^{kr}}{r!} \frac{(1)_{r}}{[(\beta + \frac{l}{k}j)_{lr}]^{n}},
\end{equation}
where in the second sum we have the Pochhammer symbols $(1)_{r}$ and $(\beta + l j/k)_{lr}$. The last of them, indexed by multiplication of integers $l$ and $r$, can be simplified into the Pochammer symbol of index $r$ according to 
\begin{equation}\label{12/06-4}
(\beta + \ulamek{l}{k}j)_{lr} = l^{r l}\prod_{i=0}^{l-1} \Big(\ulamek{\beta\, +\, l j/k}{l} + \ulamek{i}{l}\Big)_{r}. 
\end{equation}
The substitution of Eq. \eqref{12/06-4} into Eq. \eqref{6/06-2} enables {{us}} to use the series representation of the generalized hypergeometric function given by Eq. \eqref{14/03-A5} {{and obtain a formula}} involving $k$ functions ${_{1}F_{nl}}$ of the argument $z^{k}/l^{nl}$. Their upper (first) lists of parameters contain only one element equal to $1$, whereas the lower (second) lists of parameters contain $n l$ elements given by $n$ times repetition of $\Delta(l, \beta + lj/k)$ where the symbol $\Delta(r, \lambda)$ denotes the sequence of $r$ elements {{
\[
	\Delta(r, \lambda) := 
	\Bigl[ \frac{\lambda}{r}, \frac{\lambda + 1}{r}, \ldots, \frac{\lambda + r - 1}{r} \Bigr].
\]
}}
Consequently, we can express the MLR function, Eq. \eqref{6/06-2}, as 
\begin{align}\label{6/06-3}
\begin{split}
& F_{\frac{l}{k}, \beta}^{(n)}(z) = \sum_{j=0}^{k-1} \frac{z^{j}}{[\Gamma(\beta + \frac{l}{k}j)]^{n}} \\
& \times {_{1}F_{nl}}\bigg({1 \atop \underbrace{\Delta(l, \beta + \ulamek{l}{k}j), \ldots, \Delta(l, \beta + \ulamek{l}{k}j)}_{n\,\,\, \text{times}}}; \frac{z^{k}}{l^{nl}}\bigg).
\end{split}
\end{align}

\section{The MLR function $F_{\alpha, \beta}^{(n)}(-x)$ as the Laplace transform}\label{sec3}

All forms of the MLR function mentioned in the previous section satisfy the Laplace transform rule
\begin{equation}\label{24/09-1}
\int_{0}^{\infty} \E^{-st} t^{\beta-1} F_{\alpha, \beta}^{(n)}(\lambda t^{\alpha}) \D t = s^{-\beta} F_{\alpha, \beta}^{(n-1)}(\lambda s^{-\alpha}),
\end{equation}
where $\lambda$ is {{a real or complex}} constant. The proof of this formula for various representation of the MLR function can be found, e.g., in \cite{RGarra13,RGarrappa17} and also in Appendix \ref{appC}. {{After i}}nverting Eq. \eqref{24/09-1} and {{making the change of}} variable $s t = z^{1/\alpha}$ we arrive at
\begin{equation}\label{16/06-1}
F_{\alpha, \beta}^{(n)}(\lambda t^{\alpha}) = \frac{\alpha}{2\pi\!\I} \int_{L_{z}} \E^{z^{1/\alpha}} z^{\frac{1-\beta}{\alpha} - 1} F_{\alpha, \beta}^{(n-1)}(\ulamek{\lambda t^{\alpha}}{z}) \D z,
\end{equation}
where $L_{z}$ denotes the Bromwich contour with $\RE(z) > 0$. It is easy to see that both the direct and {{the}} inverse Laplace transforms in Eqs. \eqref{24/09-1} and \eqref{16/06-1}, respectively, reduce the value of $n$ in {{the}} MLR function to $n-1$. {{Thus, by using (\ref{16/06-1}) recursively one obtain nested integrals which allow to express $F_{\alpha, \beta}^{(n)}(\lambda t^{\alpha})$ in terms of $F_{\alpha, \beta}^{(1)}(\lambda t^{\alpha})$, namely}} the two-parameters ML function $E_{\alpha, \beta}(\lambda t^{\alpha})$.

\subsection{The toy model: the MRL function for $n=2$}

As a toy model we consider Eq. \eqref{16/06-1} for rational $\alpha = l/k$ ($0 < l/k < 1$), $n=2$, and $\lambda t^{l/k} = -x$. From Eq. \eqref{16/06-1} we can express the MLR function for $n=2$ as the contour integral of the MLR function for $n=1$, for which $F^{(1)}_{l/k, \beta}(-x) = E_{l/k, \beta}(-x)$. The integral form of $E_{\alpha, \beta}(-x)$ can be found in {{\cite[Eq. (6)]{KGorska18a}}} for $\gamma = 1$ which for rational $\alpha = l/k$ reads
\begin{align}\label{19/06-1}
\begin{split}
F_{\frac{l}{k}, \beta}^{(1)}(-x) & = E_{\frac{l}{k}, \beta}(-x) \\ & = \int_{0}^{\infty} \E^{-x u} u^{-1-\frac{k}{l}} g_{\frac{l}{k}, \beta}(u^{-k/l}) \frac{k  \D u}{l}. 
\end{split}
\end{align}
The auxiliary functions $g_{l/k, \beta}(\sigma)$ is given by Eq. (14) of \cite{KGorska18a}; their explicit form will be quoted later in this section. Substituting Eq. \eqref{19/06-1} into Eq. \eqref{16/06-1} for $n=2$ we have
\begin{align}\label{21/06-1}
\begin{split}
& F^{(2)}_{\frac{l}{k}, \beta}(-x) = \frac{(k/l)^{2}}{2\pi\!\I} \int_{L_{z}} \E^{z^{k/l}} z^{\frac{1-\beta}{l/k}-1} \\ & \qquad \times \left[\int_{0}^{\infty} \E^{-x u/z} u^{-1-\frac{k}{l}} g_{\frac{l}{k}, \beta}(u^{-k/l}) \D u\right]\! \D z.
\end{split}
\end{align}
Now, let $y = u/z$ be the new variable in the second integral of Eq. \eqref{21/06-1} and let us change the order of integrals. Then we end up with the Laplace transform representation of $F^{(2)}_{l/k, \beta}(-x)$ in which the weight function $m_{l/k, \beta}(2; y)$ is expressed as
\begin{align}\label{21/06-3}
\begin{split}
& m_{\frac{l}{k}, \beta}(2; y) = y^{-1 - k/l}\; \frac{(k/l)^{2}}{2\pi\!\I} \\ & \qquad \times \int_{L_{z}} \E^{z^{k/l}}\! z^{-1-\frac{k}{l}\beta} g_{\frac{l}{k}, \beta}[(y z)^{-k/l}] \D z \\
& \qquad = y^{-1 - \frac{1-\beta}{l/k}}\; \frac{k/l}{2\pi\!\I} \int_{L_{\eta}} \E^{\eta y^{-k/l}} \eta^{-1-\beta} g_{\frac{l}{k}, \beta}(\eta^{-1}) \D\eta,
\end{split}
\end{align}
where we set $\eta = (z y)^{k/l}$ and modify the contour $L_{z}$ onto $L_{\eta}$. Substituting the explicit form of the auxiliary function $g_{l/k, \beta}(\sigma)$
\begin{equation}\label{19/06-2}
g_{\frac{l}{k}, \beta}(\sigma) = \frac{\sqrt{k} l^{3/2-\beta}}{(2\pi)^{(k-l)/2}} \frac{1}{\sigma} G^{k, 0}_{l, k}\left(\frac{l^{l}}{k^{k} \sigma^{l}}\Big\vert {\Delta(l, \beta) \atop \Delta(k, 1)}\right),
\end{equation}
which for $\beta = 1$ is the one-sided L\'{e}vy stable distribution \cite{KGorska10, KGorska12}, and using the Eq. \eqref{22/06-A1} yields to 
\begin{align}\label{21/06-4}
\begin{split}
m_{\frac{l}{k}, \beta}(2; y) &= (2\pi)^{-\frac{1+k}{2} + l} \frac{k^{3/2}}{l^{2(\beta - \frac{1}{2})}} \frac{1}{y} \\ & \times G^{k, 0}_{2l, k}\left(\frac{l^{2l} y^{k}}{k^{k}}\Big\vert {\Delta(l, \beta), \Delta(l, \beta) \atop \Delta(k, 1)}\right)
\end{split}
\end{align}
and 
\begin{equation}\label{27/09-2}
2l \leq k.
\end{equation}

{
\begin{rem}
The explicit form of the auxiliary function $g_{\alpha, \beta}(\sigma)$ given by Eq. \eqref{19/06-2} for $\alpha = l/k$ comes from the Bromwich contour integral presented in Eq. (8) for $\gamma = 1$ of \cite{KGorska18a}, namely
\begin{equation}\label{27/08-1}
g_{\alpha, \beta}(\sigma) = \alpha y^{-\alpha} \int_{L} \E^{\sigma z - z^{\alpha}} \D z/(2\pi\!\I).
\end{equation}
For rational $\alpha = l/k$ it can be calculated by using \cite[Eq. (2.2.1.19)]{APPrudnikov-v5} and we get $g_{l/k, \beta}(\sigma)$ in the form of Eq. \eqref{19/06-2}. For arbitrary real $\alpha$ such that $0 < \alpha < 1$ we can employing \cite[Eq. (3)]{BStankovic70} reported in \cite[Eq. (F.2)]{FMainardi10} and we express $g_{\alpha, \beta}(\sigma)$ in the form of the generalized Wright function \eqref{12/06-4}:
\begin{equation}\label{27/08-2}
g_{\alpha, \beta}(\sigma) = \alpha \sigma^{-1-\alpha} {_{0}\Psi_{1}}\left(-\sigma^{-\alpha}\big\vert {- \atop (\beta, -\alpha)}\right).
\end{equation} 
\end{rem}
}

\subsection{The weight function $m_{l/k, \beta}(n; y)$}

Based on the considerations presented in the previous section for our toy model we can generalize Eqs. \eqref{21/06-4} and \eqref{27/09-2} thanks to the following result.
\begin{thm}\label{31/07-t1}
Let $n,l,k$ be integers such that $l/k < 1/n$. The MLR function  $F^{(n)}_{l/k, \beta}(-x)$ is the inverse Laplace transform
\begin{equation}\label{24/06-1}
F^{(n)}_{\frac{l}{k}, \beta}(-x) = \int_{0}^{\infty} \E^{-xy} m_{\frac{l}{k}, \beta}(n; y) \D y
\end{equation}
with the weight function
\begin{multline}\label{27/02-6}
m_{\frac{l}{k}, \beta}(n; y) = (2\pi)^{\frac{1-k}{2}-\frac{n(1-l)}{2}}  \frac{k^{3/2}}{l^{n(\beta-\frac{1}{2})}} \frac{1}{y} \\ \times G^{k, 0}_{nl, k}\bigg(\frac{l^{nl} y^{k}}{k^{k}}\Big\vert {\overbrace{\Delta(l, \beta), \ldots, \Delta(l, \beta)}^{n\,\, \text{times}} \atop \Delta(k, 1)} \bigg) .
\end{multline}
\end{thm}
\begin{proof}
The proof consists of two steps. First we show that the Laplace transform integral of $m_{l/k, \beta}(n; y)$ leads to the MLR function, namely (\ref{24/06-1}) and hence we prove the necessary and sufficient condition for the existence of the Laplace transform, that is 
\begin{equation}\label{7/06-15}
|m_{\alpha, \beta}(n; y)| \leq A \E^{a y}, 
\end{equation}
where $0 < y < \infty$ and $A$ as well as $a$ being constants. \\

\noindent
{\em Step 1.} We substitute $m_{l/k, \beta}(n; y)$ given by Eq. \eqref{27/02-6} into the LHS of Eq. \eqref{24/06-1}. Thus, we get the Laplace transform of the Meijer $G$ function multiplied by a power function. Such Laplace transform may be {{evaluated by}} applying Eq. \eqref{24/06-A1} in which we invert the argument according to Eq. \eqref{14/03-A3}. That gives
\begin{align}\label{28/02-1}
\begin{split}
\int_{0}^{\infty}& \E^{-x y} m_{\frac{l}{k}, \beta}(n; y) \D y= (2\pi)^{1-k - \frac{n}{2}(l-1)} \frac{k}{l^{n(\beta - \frac{1}{2})}} \\
& \times G^{k, k}_{k, k + n l}\Big(\frac{x^{k}}{l^{n l}}\Big\vert {\Delta(k, 0) \atop \Delta(k, 0), \underbrace{\Delta(l, 1-\beta)\ldots \Delta(l,1-\beta)}_{n \,\,\text{times}}}\Big)
\end{split}
\end{align}
and 
\begin{equation}\label{10/06-1}
n l \leq k. 
\end{equation}
To get the MLR function we represent the RHS of Eq. \eqref{28/02-1} as the finite sum of the hypergeometric functions which is achieved {{after}} employing Eq. \eqref{14/03-A4}. The comparison of {{the}} obtained formula with the MLR function Eq. \eqref{6/06-3} completes the first step of the proof. \\

\noindent
{\em Step 2.} Let us now find the series representation of $m_{l/k, \beta}(n; y)$. From Eq. \eqref{14/03-A4} we can express Eq. \eqref{27/02-6} as the finite sum of the hypergeometric functions:
\begin{multline}\label{28/02-2a}
m_{\frac{l}{k}, \beta}(n; y) = \sum_{j=0}^{k-1} \frac{(-y)^{j}}{j! [\Gamma(1-b_{j})]^{n}} \\ \times {_{1+nl} F_{k}}\bigg({1, \overbrace{(a_{l}), \ldots, (a_{l})}^{n\,\, \text{times}} \atop \Delta(k, 1+j)}; (-1)^{nl-k} \frac{l^{nl} y^{k}}{k^{k}}\bigg),
\end{multline}
where $(a_{l}) = \Delta(l, b_{j})$ with $b_{j} = 1- \beta + \frac{l}{k}(1 + j)$. The coefficients in Eq. \eqref{28/02-2a} can be obtained using the Gauss multiplication formula for the Gamma function
\begin{equation}\label{7/06-12}
\Gamma(n \lambda) = (2\pi)^{\frac{1-n}{2}} n^{n\lambda - \frac{1}{2}} \prod_{i=0}^{n-1}\Gamma(\lambda + \ulamek{i}{n}). 
\end{equation}
Thereafter, we substitute the series representation of ${_{1+nl}F_{k}}$ given by Eq. \eqref{14/03-A5} into Eq. \eqref{28/02-2a}. Thus, we obtain two sums: one of them is finite and contains $k$ elements and it is indexed by $j$, $j=0, 1, \ldots, k-1$, and the another one is the series over $r$, $r=0, 1, \ldots$. For these two sums we use the splitting formula given by Eq. \eqref{6/06-1} according to which the index $kr + j$ is changed into $r$, $r=0, 1, \ldots$. It yields to
\begin{multline}\label{7/06-20}
m_{\frac{l}{k}, \,\beta}(n; y) = \sum_{r=0}^{\infty} y^{r} c_{r}(\ulamek{l}{k}, \beta, n) \,\,\, \text{with} \\ c_{r}(\ulamek{l}{k}, \beta, n) = \frac{(-1)^{r}}{r! [\Gamma(1-b_{r})]^{n}}.
\end{multline}

{{Thanks to}} the series representation of $m_{l/k, \,\beta}(n; y)$ it is easy to show {{when}} condition \eqref{7/06-15} is satisfied. {{To this}} purpose we use the triangular inequality which gives the necessary and sufficient condition
\begin{multline*}
|m_{\frac{l}{k}, \beta}(n; y)| \leq \sum_{r=0}^{\infty} y^{r}|c_{r}(\ulamek{l}{k}, \beta, n)| \\ 
 \leq \frac{1}{[\Gamma(\beta - \frac{l}{k})]^{n}}\sum_{r=0}^{\infty} \frac{y^{r}}{r!} \\
 = \frac{1}{[\Gamma(\beta - \frac{l}{k})]^{n}} \E^{y}. 
\end{multline*}
\end{proof}

\begin{rem}
Eq. \eqref{27/02-6} for $l < k$, $\beta > 0$, and $n = 2$ reconstructs Eqs. \eqref{21/06-4}. For $l < k$, $\beta > 0$, and $n = 1$ we have  $k y^{-1-k/l} g_{l/k, \beta}(y^{-k/l})/l$ where $g_{l/k, \beta}(\sigma)$ is given via Eq. \eqref{19/06-2}. For $l< k$ and $n = \beta = 1$ it leads to the Meijer $G$ representation of the one-sided L\'{e}vy stable probability distribution, to see that compare $\varPhi_{l/k}(\sigma) = l \sigma^{-1-l/k} m_{l/k, 1}(1; \sigma^{-l/k})/ k$ with Eq. (2) of \cite{KGorska10}. 
\end{rem}

The weight function $m_{l/k, \beta}(n; y)$ can be also expressed through the generalized Wright function. Recalling Eq. \eqref{28/02-2a} in which we use the series form of the generalized hypergeometric function with the Gauss multiplication formula for {{the}} Gamma function we have
\begin{multline*}
m_{\frac{l}{k}, \beta}(n; y) = \sum_{j=0}^{k-1} \frac{\sin^{n}(\pi b_{j})}{\pi^{n}} (-y)^{j} \\ \times {_{1+n}\Psi_{1}}\Bigg[(-1)^{nl+k} y^{k}\Big\vert {(1, 1), \overbrace{(b_{j}, l), \ldots, (b_{j}, l)}^{n\,\,\, \text{times}} \atop (1+j, k)}\Bigg] {{,}}
\end{multline*}
{{where parameters $b_{j}$ are the same introduced immediately after Eq. \eqref{28/02-2a}.}}

\subsection{Radius of Convergence}

Below we will find the interval $y\in[0, R)$ in which the power series $m_{l/k, \beta}(n; y)$ converges, i.e., we determine its radius of convergence $R = \lim_{r\to\infty} |c_{r}/c_{r+1}|$. It is calculated from Eq. \eqref{7/06-20}  with the help of the Stirling formula $\Gamma(z) \propto \sqrt{2\pi} z^{z - 1/2} \exp(-z)$ 
and reads 
\begin{multline}\label{3/06-1a}
R = \lim_{r\to\infty} \Big\vert(1+r)\frac{\Gamma[\beta - \frac{l}{k}(2+r)]}{\Gamma[\beta - \frac{l}{k}(1+r)]}\Big\vert^{n} \\
 = \lim_{r\to\infty} \frac{1+r}{|\beta - l (1+r)/k|^{ln/k}} \\
 = (k/l)^{\frac{l}{k}n} \lim_{r\to\infty} r^{1 - \frac{l}{k}n}.
\end{multline}
From the above we see that $R$ is infinite for $l n < k$; finite for and equal to $(k/l)^{l n/k}$ for $l n = k$, and zero for $ln > k$. Thus, under the considered condition $l n \leq k$ the function $m_{l/k, \beta}(n; y)$ is well-defined. 

\section{Nonnegative character of $m_{l/k, \beta}(n; y)$}\label{sec4}

From the Bernstein theorem \cite{RLSchilling12, DVWidder46, MMerkle14} we know that any CM function can be expressed as the Laplace transform of a nonnegative weight function. In Sec. \ref{sec3} we have found the Laplace transform representation of the MLR function {{with}} $m_{l/k, \beta}(n; y)$ as {{the}} weight function. Now, we will show that $m_{l/k, \beta}(n; y)$ is nonnegative for some values of $n$ and $\beta$.  We will consider three cases: $n l = k$ is presented in the part {\bf\em I}, $2 l < k$ in the part {\bf\em II}, and the general case $nl < k$ in the part {\bf\em III}. \\

\noindent
{\bf\em I.} For $ln = k$ where $n=2, 3, \ldots$ and $l, k$ are integers such that $l < k$ the radius of convergence {{of $m_{l/k, \beta}(n; y)$}} is finite and equals to $n$. If $n$ is an integer then we can consider only the case of $l = 1$ and $k = n$. This generalizes all possible choices of $l$ and $k$ because we take integers $n$ equal to $k/l$. The weight function $m_{l/k, \beta}(n; y)$ for $l=1$, $k = n$  can be represented as
\begin{equation}\label{3/06-2}
m_{1/n, \beta}(n; y) = (2\pi)^{\frac{1-n}{2}} \frac{n^{3/2}}{y} G^{\,n, 0}_{n, n}\bigg(\frac{y^{n}}{n^{n}}\big\vert {\overbrace{\beta, \ldots, \beta}^{n \,\,\, \text{times}} \atop \Delta(n, 1)}\bigg) \Theta(n-y),
\end{equation}
where the Heaviside function $\Theta(\cdot)$ is introduced to extend the domain of integration space on the positive semiaxis $[0, \infty)$. Obviously, it encodes the information about the finite radius of convergence: $m_{1/n, \beta}(n; y)=0$ for $y > n$. 

According to the Lemma 2 of \cite{DKarp12} the Meijer G-function function appearing in the RHS of Eq.\eqref{3/06-2} is nonnegative if $\boldsymbol{\beta} = (\beta, \ldots, \beta)$ ($\beta$ occurs $n$ times) is {\em weakly supermajorized} by $\boldsymbol{n} = \Delta(n, 1)$. Due to the Definition A.2 of \cite{AWMarshall11} or to the Eq. (15) of \cite{DKarp12} the sequence $\boldsymbol{B}=(b_{1}, \ldots, b_{n})$ is weakly supermajorized by the sequence $\boldsymbol{A} = (a_{1}, \ldots, a_{n})$ if $0 < b_{1} \leq \ldots \leq b_{n}$, $0 < a_{1}  \leq \ldots \leq a_{n}$, and 
\begin{equation}\label{29/06-1}
\sum_{i=0}^{N} a_{i} \leq \sum_{i=0}^{N} b_{i} \quad \text{for} \quad N = 1, 2, \ldots, n.
\end{equation}
In our case the elements of the sequence $\boldsymbol{A}$ are the ratios $i/n$ with $i =1, 2, \ldots, n-1$, and elements of the sequence $\boldsymbol{B}$ are all equal to $\beta$. For such chosen $\boldsymbol{A}$ and $\boldsymbol{B}$ the condition Eq. \eqref{29/06-1} has the form
\begin{equation}\label{4/06-1}
\sum_{i=1}^{N}\frac{i}{n} \leq \sum_{i=1}^{N} \beta \qquad \text{for} \quad N = 1, 2, \ldots, n,
\end{equation}
from which  $\beta \geq (N+1)/(2n)$, $N=1, 2, \ldots, n$.  For the largest allowed value of $N$, i.e. $N=n$, it yields to
\begin{equation}\label{4/06-2}
\beta\geq \frac{n+1}{2 n}
\end{equation}
which provides the condition for the parameter $\beta$ under which the function $m_{1/n, \beta}(n; y)$ is nonnegative. 

\begin{exa}
According to Eq. \eqref{3/06-2} the simplest case of $m_{1/n, \beta}(n; y)$ is that for $k = n = 2$ and $\beta = 1$ which agrees with the restriction Eq. \eqref{4/06-2} giving $\beta \geq 3/4$. For the above values of $n$, $l$, $k$, and $\beta$ the function $m_{1/2, 1}(2; y)$  is
\begin{equation}\label{16/03-2}
m_{1/2, 1}(2; y) = \frac{1}{\pi}\;{_{1}F_{0}}\left({1/2 \atop -}; \frac{y^{2}}{4}\right) = \frac{2}{\pi\sqrt{4-y^{2}}}
\end{equation}
for $0 < y < 2$ and zero for $y \geq 2$. The Laplace transform of $m_{1/2, 1}(2; y)$ is known (see Eq. (2.3.7.1) of \cite{APPrudnikov-v1}) 
\begin{equation}\label{16/03-3}
\int_{0}^{2} \frac{2 \E^{-x y} \D y}{\pi\sqrt{4 - y^{2}}} = {_{0}F_{1}}\left({- \atop 1}; x^{2}\right) - \frac{4x}{\pi}\; {_{1}F_{2}}\left({1 \atop \frac{3}{2}, \frac{3}{2}}; x^{2}\right)
\end{equation}
which reproduces the MLR for $\alpha = 1/2$, $\beta = 1$, and $n=2$, see Eq. \eqref{6/06-3} for considered values of parameters $\alpha$, $\beta$, and $n$. We emphasize the importance of the condition $\beta \geq (n+1)/(2n)$; if it is broken then the function $m_{l/k, \beta}(n; y)$ may become negative, e.g. by taking $\beta = 1/2$ we arrive at
\begin{equation}\label{4/06-10}
m_{1/2, 1/2}(2; y) = -\frac{y}{4\pi} {_{1}F_{0}}\left({3/2 \atop -}; \frac{y^{2}}{4}\right) = \frac{-2 y}{\pi (4-y^{2})^{3/2}},
\end{equation}
negative for $0 < y < 2$ and 0 for $y \geq 2$. \\

The cases of $n = 3$, $n = 4$, and $n = 5$ as well as the case of $n=2$ are presented in Fig. \ref{fig00}. In Fig. \ref{fig00} the red curve presents $m_{1/n, 1}(n; y)$ for $n=2$, the blue curve is for $n=3$, the green one for $n = 4$, and the orange one is for $n=5$. All functions are defined in finite domains and for $\beta = 1$ such value of $\beta$ automatically satisfied the condition \eqref{4/06-2}.
\begin{figure}[!h]
\includegraphics[scale=0.4]{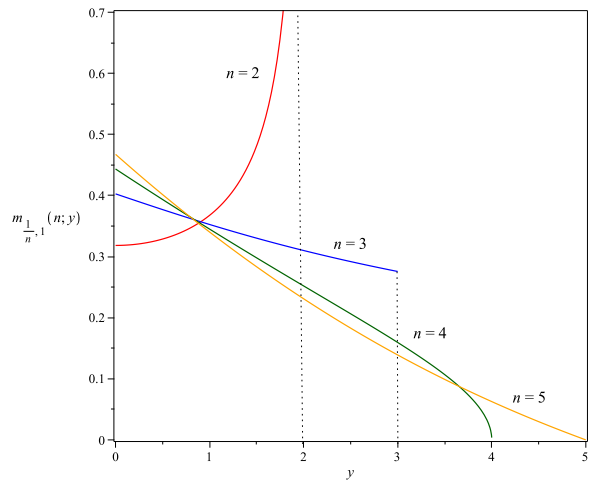}
\caption{\label{fig00} Plot of $m_{1/n, 1}(n; y)$ given by Eq. \eqref{3/06-2} for $n=2$ (the red curve), $n=3$ (the blue curve), $n=4$ (the green curve), and $n=5$ (the orange curve).}
\end{figure}
\end{exa}

\noindent
{\bf\em II.} We consider the nonnegativity of the toy model introduced  in Eq. \eqref{21/06-3}, i.e. the function $m_{l/k, 1}(2; y)$, $0 < l/k < 1/2$. The inversion of the second formula of Eq. \eqref{21/06-3},  with $\beta = 1$ and $y = u^{-l/k}$ being applied, reads
\begin{multline}\label{25/09-1}
\eta^{-2} g_{\frac{l}{k}, 1}(\eta^{-1}) = \eta^{-2} \varPhi_{\frac{l}{k}}(\eta^{-1}) \\ = \int_{0}^{\infty} \E^{-\eta u} \frac{l}{k} u^{-l/k} m_{\frac{l}{k}, 1}(2; u^{-l/k}) \D u
\end{multline}
where, as mentioned in the Section 3.1, $\varPhi_{\frac{l}{k}}(\eta^{-1})$ is the one-sided L\'{e}vy stable distribution. Next, we integrate of both side of Eq. \eqref{25/09-1} by $\exp[-(x/2)^2 \eta^{-1}]$ for $\eta >0 $. Setting $\eta^{-1} = \xi$ in the left-hand side in Eq. \eqref{25/09-1} leads to
\begin{multline}\label{12/01-1}
\int_{0}^{\infty} \E^{-(x/2)^2 \eta^{-1}} \eta^{-2} \varPhi_{\frac{l}{k}}(\eta^{-1}) \D\eta \\ = \int_{0}^{\infty} \E^{-(x/2)^2 \xi} \varPhi_{\frac{l}{k}}(\xi) \D\xi \\ = \exp[-(x/2)^{2l/k}]
\end{multline}
The {{integration}} of the right-hand side of Eq. \eqref{25/09-1} by $\int_{0}^{\infty} \E^{-(x/2)^2 \eta^{-1}} \D\eta$, where we apply $u = t^{2}${{, together with the change of }} the order of integrations{{, give }} 
\begin{multline}\label{12/01-2}
\int_{0}^{\infty} \frac{2l}{k} t^{-\frac{2 l}{k}} m_{\frac{l}{k}, 1}(2; t^{-2 l/k}) \left[\int_{0}^{\infty} \E^{-\eta t^2 - (x/2)^{2} \eta^{-1}} \D\eta\right] t \D t \\
= \sqrt{x} \int_{0}^{\infty} \frac{2 l}{k}  t^{-\frac{2 l}{k} - \frac{1}{2}} m_{\frac{l}{k}, 1}(2; t^{-2l/k}) \\ \times \sqrt{x t} K_{1}(xt) \D t.
\end{multline}
The integral in square bracket is equal to $x K_{1}(x t)/t$, where $K_{1}(xt)$ denotes the modified Bessel function of the second kind named also the McDonalds function. According to Eq. \eqref{25/09-1} we can compare Eq. \eqref{12/01-1} with \eqref{12/01-2}. The equality between these formulas yields to the Hankle-type transform \cite{KSMiller01, YuFLuchko00}:
\begin{multline}\label{13/01-1}
 x^{-\frac{1}{2}} \E^{-(x/2)^{\frac{2l}{k}}} = \int_{0}^{\infty} \frac{2 l}{k}  t^{-\frac{2 l}{k} - \frac{1}{2}} m_{\frac{l}{k}, 1}(2; t^{-2l/k}) \\ \times \sqrt{x t} K_{1}(xt) \D t.
\end{multline}
The completely monotonic character of the KWW function $\exp[-(x/2)^{2l/k}]$ for $0 < l/k < 1/2$ shown in \cite{HPollard46}. Moreover, the function with the negative power is also completely monotone. Then, their product is also completely monotonic. From {\bf Theorem 7} of \cite{KSMiller01}, see point 4, it appears that $t^{-2l/k - 1/2} m_{l/k, 1}(2; t^{-2l/k})$ is nonegative function. Because $t > 0$ then $t^{-2l/k - 1/2}$ and $m_{l/k, 1}(2; t^{-2l/k})$ are nonegative functions. That finish the proof of nonegative character of $m_{l/k, 1}(2; y)$, $y > 0$, under condition $2 l < k$.

\begin{exa}
Eq. \eqref{28/02-2a} for $\beta = l = 1$, $k = 3$, and $n=2$ gives
\begin{multline}\label{5/06-4}
m_{1/3, 1}(2; y) = \E^{-\frac{y^{3}}{27}}\left[ \frac{{_{1}F_{1}}\left({1/3 \atop 2/3}; \frac{y^{3}}{27}\right)}{[\Gamma(2/3)]^{2}} - \frac{3y [\Gamma(2/3)]^{2}}{4\pi^{2}}\right. \\ \left. \times {_{1}F_{1}}\left({2/3 \atop 4/3}; \frac{y^{3}}{27}\right)\right],
\end{multline}
where we have used Kummer's relation given by Eq. \eqref{5/06-0}. Eq. \eqref{5/06-4},  after applying to it the Eq. \eqref{5/06-1}, can be expressed in terms of the modified Bessel function of the second kind $K_{\nu}(x)$, $\nu\in\mathbb{R}$ 
\begin{equation}\label{5/06-5}
m_{1/3, 1}(2; y) = \frac{\sqrt{y}}{2\pi^{3/2}} \E^{-\frac{y^{3}}{54}} K_{1/6}(y^{3}/54).
\end{equation}
Because $K_{\nu}(x)$ is positive for $x>0$ then $m_{1/2, 1}(2; y)$ is also positive function for $0 < y < \infty$.
\end{exa}

\begin{exa} 
According to Eq. \eqref{28/02-2a} for $\beta = l = 1$, $k = 4$, and $n=2$ the function $m_{1/4, 1}(2; y)$ contains three hypergeometric functions ${_{1}F_{2}}$ which after using Eqs. \eqref{5/06-2} and \eqref{5/06-3} can be represented as
\begin{equation}\label{5/06-6}
m_{1/4, 1}(2; y) = \frac{1}{\sqrt{2}} \left[I_{1/4}(y^{4}/256) - \frac{y}{4} I_{-1/4}(y^{4}/256)\right]^{2},
\end{equation}
where $I_{\nu}(x)$'s, $\nu\in\mathbb{R}$  denote the modified Bessel function of the first kind.
Obviously, the LHS of the Eq. \eqref{5/06-6} is positive as a square of real-valued function.
\end{exa}

\noindent
{\bf\em III.} As the last case let us consider the $ln < k$ and $\beta > 0$ for which $m_{l/k, \beta}(n; y)$ is well-defined for $y\in[0, \infty)$. In this situation the asymptotics of the function$m_{l/k, \beta}(n; y)$ at $y=0$ is 
\begin{equation}\label{27/07-1}
\lim_{y\to 0} m_{l/k, \beta}(n; y) = [\Gamma(\beta - l/k)]^{-n}. 
\end{equation}
which can be deduced by taking the zero term in the series Eq. \eqref{7/06-20}. The asymptotic behavior of $m_{l/k, \beta}(n; y)$ for $y \gg 1$ can be estimated by using Eq. (7) of \cite{DKarp15} or task 1.18 on p. 41 of \cite{AMMathai10} being the special case of equation on page 289 of \cite{BLJBraaksma}. Thus, $m_{l/k, \beta}(n; y)$ for large $y$ can be approximated as 
\begin{equation}\label{3Jul2018}
m_{l/k, \beta}(n; y) \propto y^{\frac{nl - nk(\beta - 1/2)}{k-nl}} \exp\!\Big[\!-\!(k-nl)\! \left(\!\ulamek{l^{nl} y^{k}}{k^{k}}\!\right)^{\frac{1}{k-nl}} \Big],
\end{equation}
which goes to zero along the positive semi-axis. 

In this general situation the nonnegativeness of  $m_{l/k, \beta}(n; y)$ is shown by making the numerical calculation. In Figs. \ref{fig1} and \ref{fig2} are presented $m_{l/k, \beta}(n; y)$ for $l/k = 3/7$, $n=2$, $\beta = 1/2, 3/4, 1, 5/4$ and $l/k = 3/10$, $n=3$, $\beta = 1/3, 2/3, 1, 4/3$, respectively. Point out that these two examples were chosen in such a way that the condition $nl < k$ is kept. It can be observe that $m_{3/7, \beta}(2; y)$ and $m_{3/10, \beta}(3; y)$ for $\beta = 1/2$ have the negative parts for $y\in (0.086, 1.666)$ and $y\in (0.924, 3.409)$, respectively. 
\begin{figure}[!h]
\includegraphics[scale=0.4]{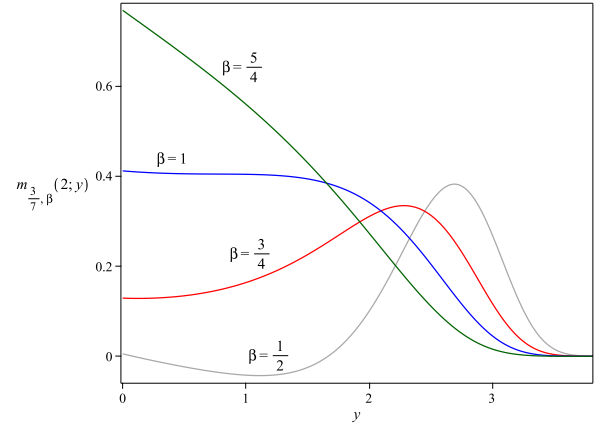}
\caption{\label{fig1} The plot of $m_{l/k, \beta}(2; y)$ for $l = 3$, $k=7$, $n=2$, and various parameter $\beta$, this is $\beta=1/2$ (the grey curve), $\beta=3/4$ (the red curve), $\beta=1$ (the blue curve), and $\beta=5/4$ (the green curve). }
\end{figure}
\begin{figure}[!h]
\includegraphics[scale=0.4]{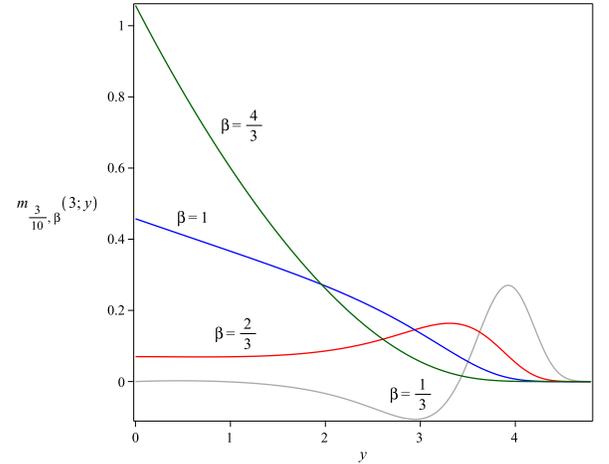}
\caption{\label{fig2} The plot of $m_{l/k, \beta}(3; y)$ for $l = 3$, $k=10$, $n=3$, and various parameter $\beta$, this is $\beta=1/3$ (the grey curve), $\beta=2/3$ (the red curve), $\beta=1$ (the blue curve), and $\beta=4/3$ (the green curve). }
\end{figure}

{{
It would be of interest, thus to give a definitive answer to the open question posed in \cite{RGarrappa17}, of finding the relationship which must satisfy $\alpha$ and $\beta$ in order to ensure the CM of the MLR function, namely the functional dependence $M_{n}(\alpha)$ such that $F_{\alpha,\beta}^{(n)}(-x)$ tuns out CM when $\beta\ge M_n(\alpha)$. For $n=1$, i.e. when $F_{\alpha,\beta}^{(1)}$ is the standard two-parameters ML function, it is known that this relationship is $\beta \ge \alpha$ with $0<\alpha \le 1$, thus $M_{1}(\alpha)=\alpha$ \cite{WRSchneider96}.

In the more general case, unfortunately, we are not able to explicitly give an analytical representation of this dependence but, thanks to the theoretical findings investigated in the paper, in particular Theorem \ref{31/07-t1}, and by means of some numerical procedures, we can approximate $M_{n}(\alpha)$ and give a graphical representation.

Indeed, a numerical procedure performed in variable precision arithmetic in Maple allow to evaluate the function $m_{\alpha, \beta}(n; y)$ in a sufficiently large interval for $y$ and for a wide range of $\alpha$ and $\beta$ and hence compute, thanks a binary search algorithm, for any $0<\alpha<1/n$ an approximation of the minimum value $M_n(\alpha)$ such that for $\beta \ge M_n(\alpha)$ it is $m_{\alpha, \beta}(n; y) \ge 0$ and hence $F_{\alpha,\beta}^{(n)}(-x)$ is CM.

The values $M_n(\alpha)$ obtained for $n = 1, 2, 3, 4$, and $5$ are presented in Fig. \ref{fig4}, where vertical dotted lines represent $1/n$ while the horizontal dotted lines are the limit values  $(n+1)/(2n)$ of Eq. (\ref{4/06-2}). The computation of each $M_n(\alpha)$ is not made on the whole interval $\alpha\in (0,1/n)$ since close to the upper bound $1/n$ the convergence of $m_{\alpha, \beta}(n; y)$ is very slow and numerically not stable.

\begin{figure}[!h]
\includegraphics[scale=0.46]{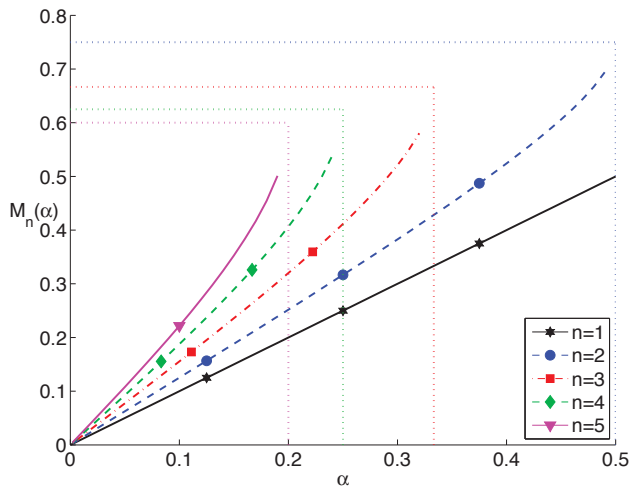}
\caption{\label{fig4} Computed bounds $M_n(\alpha)$ such that for $\beta \ge M_n(\alpha)$ it is $m_{\alpha, \beta}(n; y) \ge 0$ (and hence $F_{\alpha,\beta}^{(n)}(-x)$ is CM). }
\end{figure}

As expected, $M_n(\alpha)$ tends to be $(n+1)/(2n)$ when $\alpha \to 1/n$ and for $\alpha = 1/n$ we find the result described in Sec. 4 point {\bf \em I}. The straight line is only for $n=1$ which agrees with the already known results for the ML function.
}}

\section{\normalsize{Conclusions}}\label{sec5}

It is known that special cases of the MLR function are completely monotonic when $n=1$ and, hence, there exists their unique Laplace transform with a nonnegative weight function. Inspired by this fact we expressed the MLR function $F_{\alpha, \beta}^{(n)}(\cdot)$, $n =1, 2, \ldots$ and rational $\alpha = l/k$ such that $0 < l/k < 1$, as the Laplace transform of the weight function $m_{\alpha, \beta}(n; -)$. Thereafter, for proving the nonnegativity of $m_{\alpha, \beta}(n; -)$ we considered three cases: (i) $ln = k$, (ii) $2l < k$ and $\beta = 1$, and (iii) $ln < k$. In the first two cases, namely (i) and (ii), we were able to prove the completely monotonicity of  $F_{1/n, \beta}^{(n)}(-x)$ and $F_{l/k, 1}^{(n)}(-x)$.

From the weakly supermajorized theorem \cite{DKarp12, AWMarshall11} appears that the weight function $m_{1/n, \beta}(n; y)$ is nonnegative for $\beta \geq (n+1)/(2n)$. This claim is in agreement with the fact that $m_{1/n, \beta}(n; y)$ has a finite radius of convergence which means that we have to deal with the Laplace transform  defined on a finite sector. Such Laplace transform corresponds to the Hausdorff moment problem (the moment problem defined in the finite sector) which contains the positive defined weight function \cite{DVWidder46} related to $m_{1/n, \beta}(n; y)$. In the case (ii) with the help of {\bf Theorem 7} point 4 of \cite{KSMiller01} we were able to prove the nonnegativity of  $m_{l/k, \beta}(n; y)$ for $n = 2$ and $\beta = 1$. Here, we specify two examples $l/k = 1/3$ and $l/k = 1/4$ for which the weight function $m_{l/k, \beta}(n; y)$ can be presented as the positive standard function which contains the Bessel function of the second kind and the square of the difference of the Bessel functions of the first kind. The last case, i.e. $ln < k$, is more difficult. For this case we were able to consider only the asymptotics of $m_{l/k, \beta}(n; y)$ and show their nonnegative behaviour 
numerically. 

We believe that for the proof of the nonnegative character of $m_{l/k, \beta}(n; y)$ it will be helpful its representation in terms of nested integrals. This form can be derived by using the Laplace transform of the MRL function. Every times when we apply Eq. \eqref{16/06-1} we reduce the parameter $n$ of the MLR function by one such that $n$ times used this formula allow one to present the $n$ MLR function as the nested integrals which started with $F_{\alpha, \beta}^{(1)}(-x)$, i.e. the two-parameters ML function $E_{\alpha, \beta}(-x)$. This two-parameters ML function is given by Eq. \eqref{19/06-1} and it is defined through $g_{\alpha, \beta}(u)$ function. The integral representation of $g_{l/k, \beta}(u)$ can be find in \cite{KGorska18a}, see Eq. (7) for $\gamma=1$ and rational $\alpha = l/k$: 
\begin{equation}\label{26/09-1}
g_{\frac{l}{k}, \beta}(u) = \frac{u^{-1-l/k}}{2\pi\!\I} \int_{L_{\tilde{z}}} \E^{z^{k/l}} \E^{-z u^{-l/k}} z^{\frac{1-\beta}{l/k}} \D z.
\end{equation}
Substituting Eq. \eqref{26/09-1} into $n$ times used Eq. \eqref{16/06-1} in which $\lambda t^{l/k} = -x$ we can obtain that 
\begin{multline}\label{26/09-2}
F_{\frac{l}{k}, \beta}^{(n)}(-x) = \int_{0}^{\infty} \E^{-x y} \left(\frac{k/l}{2\pi\!\I} \int_{L_{\xi_{n}}} \D\xi_{n} \E^{\xi_{n}^{k/l}} \xi_{n}^{\frac{1-\beta}{k/l}} \times \ldots\right. \\ \left. \times\;  \frac{k/l}{2\pi\!\I} \int_{L_{\xi_{1}}} \D\xi_{1} \E^{\xi_{1}^{k/l}} \xi_{1}^{\frac{1-\beta}{k/l}} \E^{-\xi_{n}\ldots \xi_{1} y}\right) \D y.
\end{multline}
The Bromwich contour $L_{\xi_{j}}$, $j=1, 2, \ldots, n$, is with $\Re(\xi_{i}) > 0$. Eq. \eqref{26/09-2}  can be proved by induction. The circle bracket defined the weight function $m_{l/k, \beta}(n; y)$. 

\section*{\normalsize{Acknowledgments}} 

K.G and A.H. were supported by the NCN, OPUS-12, program no. UMO-2016/23/B/ST3/01714. Moreover, K.G. thanks for support from NCN (Poland), Miniatura 1, program no. 2017/01/X/ST3/00130. The work of R.G. is supported under the Cost Action CA 15225.

{The authors would like to thank an anonymous referee for the drawing the attentions to Refs. \cite{FMainardi10, BStankovic70}.}

\appendix
\renewcommand{\theequation}{\Alph{section}.\arabic{equation}}

\section{The Meijer $G$ function} \label{appA}

The Fox $H$ function and its special case the Meijer $G$ function \cite{APPrudnikov-v3} are defined as an inverse Mellin-Barnes transform as follows: the Fox $H$ function as
\begin{multline}\label{2/08-1}
H^{m, n}_{p, q} \left[z\Big\vert {[a_{p}, A_{p}] \atop [b_{q}, B_{q}]}\right] \okr \frac{1}{2\pi\!\I}\int_{\gamma_{L}}\!\!\! \D s\; x^{-s}\\ \times \frac{\prod_{i=1}^{m}\Gamma(b_{i} + B_{i}s) \prod_{i=1}^{n}\Gamma(1-a_{i}-A_{i}s)}{\prod_{i=n+1}^{p}\Gamma(a_{i}+A_{i}s) \prod_{i=m+1}^{q}\Gamma(1-b_{i}-B_{i}s)},
\end{multline}
and if we take $A_{i} = 1$, $i=1, 2, \ldots, p$, as well as $B_{j} = 1$, $j=1, 2, \ldots, q$, we have the Meijer $G$ function
\begin{equation}\label{14/03-A1} 
H^{m, n}_{p, q} \left[z\Big\vert {[a_{p}, 1] \atop [b_{q}, 1]}\right] = G^{m, n}_{p, q} \left(z\Big\vert {(a_{p}) \atop (b_{q})}\right) 
\end{equation}
where empty products are taken to be equal to one. In Eqs. \eqref{2/08-1} and \eqref{14/03-A1} the parameters are subject of conditions
\begin{align*}
\begin{split}
& z\neq 0, \quad 0 \leq m \leq q, \quad 0 \leq n \leq p; \\
& a_{i}\in\mathbb{C}, \,\,\, A_{i} > 0, \quad i=1, \ldots, p; \\
& b_{i}\in\mathbb{C}, \,\,\, B_{i} > 0, \quad i=1, \ldots, q; \\
& [a_{p}, A_{p}] = (a_{1}, A_{1}), \ldots, (a_{p}, A_{p}); \\ 
& [b_{q}, B_{q}] = (b_{1}, B_{1}), \ldots, (b_{q}, B_{q}); \\
& (a_{p}) = a_{1}, a_{2}, \ldots, a_{p}; \quad (b_{q}) = b_{1}, b_{2}, \ldots, b_{q}.
\end{split}
\end{align*}
For a full description of integration contour $\gamma_{L}$, several properties and special cases of the $G$ and $H$ function see \cite{APPrudnikov-v3}. 

Below we quote the explicit formulas of some properties of the Meijer $G$ function which are widely used in the paper:   

\noindent
{\bf (-)} from Eq. (8.2.2.14) of \cite{APPrudnikov-v3} we have the formula  
\begin{equation}\label{14/03-A3}
G^{m, n}_{p, q}\left(z\Big\vert {(a_{p}) \atop (b_{q})}\right) = G^{n, m}_{q, p}\left(\frac{1}{z}\Big\vert {1- (b_{q}) \atop 1-(a_{p})}\right),
\end{equation}
which invert the argument of the Meijer $G$ function; 

\noindent
{\bf (-)}  the formula transforming the Meijer $G$ function into the generalised hypergeometric function for $p \leq q$ has the form of Eq.(8.2.2.3) of \cite{APPrudnikov-v3} and looks like
\begin{align}\label{14/03-A4}
\begin{split}
&G^{m, n}_{p, q}\left(z\Big\vert {(a_{p}) \atop (b_{q})}\right) \\ & = \sum_{j=1}^{m} \frac{\Big[\prod_{i=0}^{n-1}\Gamma(b_{i}-b_{j})\Big]' \prod_{i=0}^{n-1}\Gamma(1+b_{j}-a_{i})}{\prod_{i=n+1}^{p}\Gamma(a_{i}-b_{j}) \prod_{i=m+1}^{q}\Gamma(1+b_{j}-b_{i})} \\
& \times z^{b_{j}} {_{p}F_{q-1}} \left({1 + b_{j} - (a_{p}) \atop 1+b_{j} - (b_{q}) ''}; (-1)^{p-m-n} z\right), \\
\end{split}
\end{align}
where $b_{i} - b_{j} \neq 0, \pm 1, \ldots$ for $i\neq j$, $i, j = 1, 2, \ldots, m$; \\ $1 + b_{j} - (b_{q})'' = 1 + b_{j} - b_{1}, \ldots, 1 + b_{j} - b_{j-1}, 1+b_{j}-b_{j+1}, \ldots, 1+b_{j}-(b_{q})$; and $\Big[\prod_{i=0}^{n-1}\Gamma(b_{i}-b_{j})\Big]' = \prod_{i=0}^{j-1} \Gamma(b_{i}-b_{j}) \prod_{i=j+1}^{n-1} \Gamma(b_{i}-b_{j})$. \\

\noindent
{\bf (-)} The inverse Laplace transform of Meijer $G$ function is given by Eq. (3.38.1.1) of \cite{APPrudnikov-v5}, namely
\begin{align}\label{22/06-A1}
\begin{split}
\frac{1}{2\pi\!\I}& \int_{L}  \E^{\,\sigma\,x} \sigma^{-\lambda} G^{m, n}_{p, q} \left(\omega p^{l/k}\Big\vert {(a_{p}) \atop (b_{q})}\right) = \frac{(2\pi)^{\frac{l-1}{2} - c(k-1)} k^{\mu}}{l^{\lambda-1/2} x^{1-\lambda}} \\ 
& \times G^{k m, k n}_{k p + l, k q} \left(\frac{\omega^{k} l^{l}}{k^{k(q-p)} x^{l}}\Big\vert {\Delta(k, (a_{p})), \Delta(l, \lambda) \atop \Delta(k, (b_{q}))}\right),
\end{split}
\end{align}
where $L$ is the Bromwich contour, $\mu = \sum_{j=1}^{q} b_{j} - \sum_{j=1}^{p} a_{j} + (p-q)/2 + 1$, and $c = m+n -(p+q)/2$. 

\noindent
{\bf (-)}  The Laplace integration of Meijer $G$ function is given in Eq. (2.24.3.1) of \cite{APPrudnikov-v3} 
\begin{align}\label{24/06-A1}
\begin{split}
\int_{0}^{\infty}& \E^{- \sigma x} x^{\alpha-1} G^{m, n}_{p, q}\left(\omega x^{l/k}\Big\vert {(a_{p}) \atop (b_{q})}\right) \D x = \frac{k^{\mu} l^{\alpha-\frac{1}{2}} \sigma^{-\alpha}}{(2\pi)^{\frac{l-1}{2}+ c(k-1)}} \\
& \times G^{k m, k n+l}_{k p + l, k q}\left(\frac{\omega^{k} l^{l}}{\sigma^{l} k^{k(q-p)}}\Big\vert {\Delta(l, 1-\alpha), \Delta(k, (a_{p})) \atop \Delta(k, (b_{q}))}\right),
\end{split}
\end{align}
where $c$ and $\mu$ are introduced below Eq. \eqref{22/06-A1}. Here we quote only this conditions which appeared in the considered case: $p \leq q$ and $c \geq 0$.

\section{The generalised hypergeometric function}\label{appB}

The generalised hypergeometric function ${_{p}F_{q}}\Big({(c_{p}) \atop (d_{q})}; x\Big)$, $x\in\mathbb{R}$, is defined via the series \cite{APPrudnikov-v3}:
\begin{equation}\label{14/03-A5}
{_{p}F_{q}}\left({(c_{p}) \atop (d_{q})}; x\right) \okr \sum_{r=0}^{\infty} \frac{x^{r}}{r!} \frac{(c_{1})_{r} (c_{2})_{r}\cdots (c_{p})_{r}}{(d_{1})_{r} (d_{2})_{r}\cdots (d_{q})_{r}},
\end{equation}
where $(c)_{r}$ is the Pochhammer symbol (rising factorial) given by $\Gamma(c + r)/\Gamma(c)$. The empty Pochhammer symbol is equal to one. \\

Below we itemize some properties which are used in the paper: \\
\noindent
{\bf (-)} the cancelation formula given by Eq. (7.2.3.7) of \cite{APPrudnikov-v3}, according to which the same terms in nominator and in denominator can be cancelled: 
\begin{equation}\label{7/06-2}
{_{p}F_{q}}\left({(a_{p-r}), (c_{r}) \atop (b_{q-r}), (c_{r})}; z\right) = {_{p-r}F_{q-r}}\left({(a_{p-r}) \atop (b_{q-r})}; z\right);
\end{equation} 

\noindent
{\bf (-)} Kummer's relation transforming ${_{1}F_{1}}$ into another ${_{1}F_{1}}$: 
\begin{equation}\label{5/06-0}
{_{1}F_{1}}\left({a \atop b}; z\right) = \E^{z} {_{1}F_{1}}\left({b-a \atop b}; -z\right),
\end{equation}
see Eq. (7.11.1.2) of \cite{APPrudnikov-v3}. \\

Because of in the text we extensively applied relations between the generalised hypergeometric function and the special function we list them below: \\
\noindent
{\bf (-)} combining Eq. (7.11.1.21) for $b = 2a$ of \cite{APPrudnikov-v3} with Eq. (7.11.4.5) of \cite{APPrudnikov-v3} we get 
\begin{multline}\label{5/06-1}
\qquad \frac{\Gamma(1-2a)}{\Gamma(1-a)} {_{1}F_{1}}\left({a \atop 2a}; z\right) \\ + \frac{\Gamma(2a-1)}{\Gamma(a)} z^{1-2a} {_{1}F_{1}}\left({1-a \atop 2(1-a)}; z\right) \\
= \frac{z^{1/2-1}}{\sqrt{\pi}} \E^{z/2} K_{a - 1/2}(z/2)
\end{multline}
with $K_{\nu}(\sigma)$ being the modified Bessel function of the second kind; and 

\noindent
{\bf (-)} for special chosen lists of upper and lower of parameters in ${_{1}F_{2}}$ can be transform into the modified Bessel function of the first kind $I_{\nu}(\sigma)$, $\nu\in\mathbb{R}$. The use of Eq. (7.14.1.7) of \cite{APPrudnikov-v3} gives
\begin{equation}\label{5/06-2}
{_{1}F_{2}}\left({a \atop a+1/2, 2a}; z\right) = [\Gamma(a + \ulamek{1}{2})]^{2} \left(\frac{z}{2}\right)^{1-2a} I_{a-1/2}^{2}(z),
\end{equation}
and Eq. (7.14.1.9) of \cite{APPrudnikov-v3} reads
\begin{equation}\label{5/06-3}
{_{1}F_{2}}\left({1/2 \atop b, 2-b}; z\right) = \frac{\pi(1-b)}{\sin(b\pi)} I_{1-b}(z) I_{b-1}(z);
\end{equation}

\section{The proof that Eq. \eqref{6/06-3} satisfies Eq. \eqref{24/09-1}} \label{appC}

The Laplace transform of $t^{\beta-1} F_{l/k, \beta}^{(n)}(\lambda t^{\alpha})$, where $\lambda$ is a complex or real constant and the MLR function is given via Eq. \eqref{6/06-3}, can be written as
\begin{align}\label{7/06-1}
\begin{split}
&\int_{0}^{\infty} \E^{-s t} t^{\beta-1} F_{\frac{l}{k}, \beta}^{(n)}(\lambda t^{l/k}) \D t = \sum_{j=0}^{k-1} \frac{\lambda^{j}}{[\Gamma(\beta + \frac{l}{k}j)]^{n}} \\
&\;\; \times \int_{0}^{\infty}\!\! \E^{-s t} t^{\frac{l}{k}j + \beta-1} \\ & \;\; \times {_{1}F_{nl}}\bigg({1 \atop \underbrace{\Delta(l, \beta + \ulamek{l}{k}j), \ldots, \Delta(l, \beta + \ulamek{l}{k}j)}_{n\,\,\, \text{times}}}; \frac{\lambda^{k} t^{l}}{l^{nl}}\bigg),
\end{split}
\end{align}
where we changed the order of the integral and the finite sum. The integral in the RHS of Eq. \eqref{7/06-1} can be calculated expliticely by employing Eq. (7.525.1) of \cite{ISGradshteyn07}. Due to it we get the generalised hypergeometric function of the type ${_{l+1}F_{nl}}$ with the upper list of parameters contains $1$ and $\Delta(l, \beta + \frac{l}{k}j)$, and the lower list of parameters with $n l$ elements, namely $n$ times repeated $\Delta(l, \beta + \frac{l}{k}j)$. Then, according to Eq. \eqref{7/06-2}, we cancel the same one term $\Delta(l, \beta + \frac{l}{k}j)$ from upper and lower list of parameters. That yields to the know formula being the Laplace transform of the MLR function, this is
\begin{align}\label{7/06-3}
\begin{split}
& \int_{0}^{\infty}\!\! \E^{-s t} t^{\beta-1} F_{\frac{l}{k}, \beta}^{(n)}(\lambda t^{l/k}) \D x = \sum_{j=0}^{k-1} \frac{(\lambda s^{-\frac{l}{k}})^{j}}{[\Gamma(\beta + \frac{l}{k}j)]^{n-1}} \\
& \times {_{1}F_{(n-1)l}}\bigg({1 \atop \underbrace{\Delta(l, \beta + \ulamek{l}{k}j), \ldots, \Delta(l, \beta + \ulamek{l}{k}j)}_{n-1\,\,\, \text{times}}}; \frac{(\lambda s^{-l/k})^{k}}{l^{(n-1)l}}\bigg) \\
& = s^{-\beta} F_{\frac{l}{k}, \beta}^{(n-1)}(\lambda s^{-l/k}).
\end{split}
\end{align}
The obtained equality can be treated as another proof of the Laplace transform presented in Eq. (3.1) of \cite{RGarrappa17} or Eq. (3.2) of \cite{RGarra13}.

\bigskip \smallskip

 \it

 \noindent

$^{a}$ H. Niewodnicza\'{n}ski Institute of Nuclear Physics, Polish Academy of Sciences, 
ul. Eljasza-Radzikowskiego 152, PL 31342 Krak\'{o}w, Poland  \\ [4pt]
e-mail: katarzyna.gorska@ifj.edu.pl, andrzej.horzela@ifj.edu.pl \\[6pt]

$^{b}$ Department of Mathematics, University of Bari "Aldo Moro", Via Orabona n. 4 - 70125 Bari, Italy -
{{Member of the INdAM Research group GNCS}}\\[4pt]
e-mail:  roberto.garrappa@uniba.it



\end{document}